\documentclass[pdflatex,sn-mathphys-num]{sn-jnl}

\usepackage{graphicx}%
\usepackage{multirow}%
\usepackage{amsmath,amssymb,amsfonts}%
\usepackage{amsthm}%
\usepackage{mathrsfs}%
\usepackage[title]{appendix}%
\usepackage[font=scriptsize,labelfont=bf]{caption}
\usepackage{xcolor}%
\usepackage{textcomp}%
\usepackage{manyfoot}%
\usepackage{booktabs}%
\usepackage{listings}%

\usepackage{tikz}
\usepackage{diagbox}
\usepackage{float}  
\usepackage{booktabs}
\usepackage{subcaption}       
\usepackage{xcolor}
\usepackage{amsmath, amssymb, mathtools}

\usepackage[title]{appendix}%

\usepackage{optidef}
\usepackage[ruled,vlined]{algorithm2e}
\SetKwInput{KwInit}{Initialize}
\usepackage[normalem]{ulem}   

\newcommand{\Lc}{\ensuremath{\mathcal{L}}}

\newcommand{\RR}{\ensuremath{\mathbb{R}}}

\newcommand{\Nc}{\ensuremath{\mathcal{N}}}
\newcommand{\NN}{\ensuremath{\mathbb{N}}}

\newtheorem{theorem}{Theorem}
\newtheorem{claim}{Claim}

\newtheorem{assumption}{Assumption}

\newcommand{\bu}{\mathbf{u}}
\newcommand{\bv}{\mathbf{v}}
\newcommand{\bw}{\mathbf{w}}
\newcommand{\bfu}{\mathbf{f}}
\newcommand{\blam}{\boldsymbol{\lambda}}

\begin{document}

\title[Article Title]{An Alternating Direction Method of Multipliers for Topology Optimization}


\author*[1]{\fnm{Harsh} \sur{Choudhary}}\email{choudhar@fel.cvut.cz}

\author[2]{\fnm{Sven} \sur{Leyffer}}\email{leyffer@anl.gov}

\author[2]{\fnm{Dominic} \sur{Yang}}\email{dominic.yang@anl.gov}

\affil*[1]{\orgdiv{Computer Science}, \orgname{Czech Technical University in Prague}, \orgaddress{\street{13 Charles Square}, \city{Prague}, \postcode{12000}, \state{Prague}, \country{Czechia}}}

\affil[2]{\orgdiv{Mathematics and Computing}, \orgname{Argonne National Lab}, \orgaddress{\street{9700 S Cass Ave}, \city{Lemont}, \postcode{60493}, \state{Illinois}, \country{USA}}}


\abstract{We consider a class of integer-constrained optimization problems governed by partial differential equation (PDE) constraints and regularized via total variation (TV) in the context of topology optimization. The presence of discrete design variables, nonsmooth regularization, and non-convex objective renders the problem computationally challenging. To address this, we adopt the alternating direction method of multipliers (ADMM) framework, which enables a decomposition of the original problem into simpler subproblems that can be solved efficiently. The augmented Lagrangian formulation ensures consistency across variable updates while facilitating convergence under appropriate conditions.}


\vspace{-0.2cm}
\keywords{Mixed-Integer Nonlinear Optimization, PDE Constrained Optimization, Topology Optimization}



\maketitle
\vspace{-0.5cm}
\section{Problem Statement}\label{s:prob_state}
\vspace{-0.2cm}
Topology optimization (TO) is a method for producing optimal design patterns within a given domain, subject to specific constraints. TO finds applications in fields like engineering, material science, thermodynamics etc. In practice, TO problems are ill-posed due to rapidly oscillating material distributions or fine-scale structures that are not physically realizable\cite{clason2018total} so we usually add a TV regularizer term in our optimization objective.
The specific problem we are interested in is the design of an optimal heat sink where we have two materials of given conductivity and the total conductivity of the material is given by a linear combination of these. We seek an optimal distribution of a discrete design variable $w(x)$ where $x$ in some domain $\Omega$ so as to minimize the objective in \eqref{eq:original_prob}, which consists of two terms, a compliance and a total-variation regularizer. The constraints include the PDE, boundary condition, and a volume constraint that limits the volume of $w(x)$ in $\Omega$ to a volume fraction $V_{\max}$:
\vspace{-0.2cm}
\begin{subequations}\label{eq:original_prob}
\begin{align}
\inf_{u, w} \quad  
&J(w) = \int_{\Omega} \mathcal{F} \cdot u(x, w(x)) \, dx 
+ \alpha \int_{\Omega} |\nabla w(x)| \, dx \tag{1.a} \label{eq:original_prob_a} \\
\text{s.t.} \quad
&\nabla \cdot (k(w(x)) \nabla u(x)) + \mathcal{F}(x) = 0, 
\quad \text{in } \Omega \tag{1.b} \label{eq:original_prob_b} \\
&-k(w(x)) \nabla u(x) \cdot n = 0, 
\quad \text{on } \partial \Omega_N \tag{1.c} \label{eq:original_prob_c} \\
&k(w(x)) = k_0 (1 - w(x)) + k_1 w(x) \tag{1.d} \label{eq:original_prob_d} \\
&u(x) = 0, 
\quad \text{on } \partial \Omega_D \tag{1.e} \label{eq:original_prob_e} \\
&\int_{\Omega} w(x) \, dx \leq V_{\max} \tag{1.f} \label{eq:original_prob_f} ; \quad w(x) \in \{0, 1\}, 
\quad \forall x \in \Omega,
\end{align}
\end{subequations}
where $\partial\Omega_D$ and $\partial\Omega_D$ denote Dirichlet and Neumann boundaries, respectively.

We discretize \eqref{eq:original_prob} to obtain a finite-dimensional optimization problem. We denote by $E$, the set of finite elements, and we use piecewise constant elements to discretize the control $w$ and let $w_e$ denote the value of $w(x)$ on element $e\in E$. Likewise, $u_e$ denotes the nodal values of $u$ on $e\in E$ and $\Nc(e), e\in E$ denotes the set of adjacent elements of $e$. 
To make the TV norm independent of the type of discretization, we introduce edge-weights, $s_{ee'}$, for all adjacent elements.
To distinguish the resulting finite-dimension problem from \eqref{eq:original_prob}, we introduce vectors $(\bu,\bw)$ to denote the discretized states, $u_e$ and controls, $w_e$. 
\vspace{-0.2cm}
\begin{equation}\label{eq:finite_dim}
\begin{aligned} 
\underset{\bu,\, \bw}{\text{min}} \quad 
& J(\bu,\bw) = \sum_{e \in E} f_e^T u_e + \alpha \sum_{e \in E} \sum_{e' \in \mathcal{N}(e)} |w_e - w_{e'}| \cdot s_{ee'} \\
\text{s.t.} \quad 
& A(\bw)\bu + \bfu = 0, \quad \text{including B.C.} \\
& \sum_{e \in E} w_e \leq V_{\max} \cdot |E|, \qquad \bw \in \{0, 1\}^{|E|}
\end{aligned}
\end{equation}

where we have combined the PDE and boundary conditions, \eqref{eq:original_prob_b}-\eqref{eq:original_prob_e}, into $A(w)u + f = 0$. The discretized problem is a mixed integer nonlinear optimization problem and is hard to solve because the feasible set is nonconvex with integer variables. Another challenges being non-smoothness of the TV term and non-convexity of compliance as a function of $\bw$.

\vspace{-0.2cm}
\section{Solution Methodology}
\vspace{-0.2cm}
To solve \eqref{eq:finite_dim}, we propose an alternating direction method of multipliers (ADMM). We first introduce an auxiliary variable $v \in [0, 1]$, a continuous relaxation of the discrete control $w$ and obtain an equivalent problem 
\eqref{eq:aug_prob}:
\vspace{-0.2cm}
\begin{equation}\label{eq:aug_prob}
\begin{aligned}
\underset{\bu, \bw, \bv}{\text{min}}{} \quad & J(\bu,\bv,\bw) = \sum_{e \in E} f_e^T u_e + 
\alpha \sum_{e \in E} \sum_{e' \in \mathcal{N}(e)} |w_e - w_{e'}| \cdot s_{ee'} \\
\text{s.t.} \quad & A(\bv)\bu + \bfu = 0,  \\
& \bv - \bw = 0 \\
& \sum_{e \in E} w_e \leq V_{\max}\cdot|E| \ ; \quad \bw \in \{0, 1\}^{|E|}\\
& \sum_{e \in E} v_e \leq V_{\max}\cdot|E| \ ;\quad \bv \in [0, 1]^{|E|},
\end{aligned}
\end{equation}
where we also use the SIMP formulation of conductivity \cite{krishna2017topology}, because it promotes the integrality of $v$ : $k(v) = \delta + (1-\delta)v^p$.

The bold symbols in \eqref{eq:aug_prob} and subsequent equations represent vectors. 
Note that we have also duplicated the volume constraint. This problem decouples into two independent subproblems if not for the copy constraint $\bv = \bw$. To remove the copy constraint, we dualize the constraint and construct the augmented Lagrangian.
\begin{equation}\label{eq:aug_lagr}
\begin{aligned} 
    \Lc_{\rho}(\bu, \bv, \bw, \blam) = \bfu^T \bu
 + \alpha \sum\limits_{e \in E} \sum\limits_{e' \in \Nc(e)} |w_e - w_{e'}| \cdot s_{ee'}
 + \frac{\rho}{2} \| \bw - \bv + \blam \|^2_2 -\frac{\rho}{2}\|\blam\|^2_2 
\end{aligned}
\end{equation}

ADMM alternates between two subproblems: a continuous one in $\bv$ and a discrete one in $\bw$. The first subproblem \eqref{eq:subprob1} is a nonlinear (discretized) PDE-constrained optimization problem for the continuous control $\bv$. 
\begin{equation}\label{eq:subprob1}
\texttt{Cont}(\bv; \overline{\bw}, \overline{\blam}) \;\triangleq\;
\left\{
\begin{aligned}
\underset{\textbf{u, v}}{\text{min}}{} \quad & \Lc_{\rho}(\bu, \bv; \overline{\bw}, \overline{\blam}) := \bfu^T\bu +  \frac{\rho}{2}\|\overline{\bw} - \bv - \overline{\blam}\|^2_2\\
\text{s.t.} \quad & A(\bv)\bu + \bfu = 0, \quad  \\
&\sum_{e \in E} v_e \leq V_{\max} \cdot |E| ; \quad \bv \in [0, 1]^{|E|}
\end{aligned}
\right.
\end{equation}
In practice, we solve \eqref{eq:subprob1} by a reduced-space approach that eliminates the discretized states, $\bu$ using the controls $\bv$ in the discretized PDE. We solve this problem using some gradient-based constraint optimization method like IPOPT \cite{wachter2006implementation}. 

The second subproblem \eqref{eq:subproblem2} is interesting in the sense that we have a TV regularizer along with the quadratic penalty term and a budget constraint. This resembles a class of problems presented in \cite{yang2025specialized}, which finds applications in image denoising, topology optimization, etc. These types of problems can be solved using an augmentation scheme where a solution is incremented or decremented on connected subregions. We follow the discussion in \cite{yang2025specialized} to randomize these moves and attain an approximating heuristic algorithm for this subproblem.

\begin{equation}\label{eq:subproblem2}
\texttt{Disc}(\mathbf{w}, \overline{\mathbf{v}}, \overline{\blam}) \;\triangleq\;
\left\{
\begin{aligned}
\underset{\mathbf{w}}{\text{min}}\quad &\mathcal{L}_{\rho}(\mathbf{w}; \overline{\mathbf{v}}, \overline{\blam}) := 
\alpha \sum\limits_{e \in E} \sum\limits_{e' \in \Nc(e)} |w_e - w_{e'}| \cdot s_{ee'}
 + \frac{\rho}{2} \|
 \mathbf{w} - \overline{\mathbf{v}} - \overline{\blam}\|_2^2 \\
\text{s.t.} \quad & \sum_{e \in E} w_e \leq V_{\max} \cdot |E| ;  \quad \bw \in \{0,1\}^{|E|}
\end{aligned}
\right.
\end{equation}
\vspace{-0.2cm}
where $s_{ee'} \in \{1, \sqrt{2}\}$ are the edge weights. We can now define our ADMM algorithm.
\vspace{-0.2cm}

      
      

\begin{algorithm}[h!]
  \caption{ADMM for Topology Optimization\label{alg:admm_TO}}
  \KwInit{$\bv^{(0)}, \bw^{(0)}, \blam^{(0)}, j=0$, set $\tau^{(0)} = \max(1, \|\bw^{(0)} - \bv^{(0)}\|^2_2)\,\gamma$} 
  \While{$\|\bw^{(j)} - \bv^{(j)}\|_2^2 > \delta(E)$}{
      $\bv^{*} \leftarrow \arg\min_\bv \texttt{Cont}(\mathbf{v}; {\mathbf{w}^{(j)}}, {\blam^{(j)}})$ \tcp*[r]{cont. subprob.}
      
      $\bw^{*} \leftarrow \arg\min_\bw \texttt{Disc}(\mathbf{w}; {\mathbf{v}^{(j+1)}},\blam^{(j)})$ \tcp*[r]{discr. subprob.}
      
    \eIf{$\|\bw^{*} - \bv^{*}\|^2_2 \le \beta\,\tau^{(j)}$}{
      $\tau^{(j+1)} \gets (1-\zeta)\tau^{(j)} + \zeta (\|\bw^{*} - \bv^{*}\|^2_2)$\; $\bv^{(j+1)}\gets\bv^{*};\quad \bw^{(j+1)}\gets\bw^{*}$\tcp*[r]{accept step}
      $\blam^{(j+1)} \gets \blam^{(j)} + \rho^{(j)}(\bw^{(j+1)} - \bv^{(j+1)})$ \tcp*[r]{update multiplier}
    }{
      $\bv^{(j+1)}\gets\bv^{(j)};\quad \bw^{(j+1)}\gets\bw^{(j)}$\tcp*[r]{reject step}
      $\rho^{(j+1)} \gets c\,\rho^{(j)}$ \tcp*[r]{update penalty}
    }
    $j \gets j+1$\;
    
  }
\end{algorithm}
\vspace{-0.2cm}
Unlike traditional ADMM methods that use a fixed penalty parameter, we adopt the funnel-based strategy in \cite{kiessling2024unified} to adaptively increase the penalty when constraint infeasibility is high. Let the tolerance threshold be $\beta \tau^{(j)}$, with $\tau^{(0)} = \max(1, \|\bw^{(0)} - \bv^{(0)}\|^2_2)\gamma$ and some constant $\gamma > 1$. If the computed step is such that  $\|\bw^{(j+1)} - \bv^{(j+1)}\|^2_2 \le \beta \tau^{(j)}$, the step is accepted and the threshold is tightened via $\tau^{(j+1)} \gets (1 - \zeta)\tau^{(j)} + \zeta (\|\bw^{(j+1)} - \bv^{(j+1)}\|^2_2)$, for $0 < \zeta < 1$. Otherwise, the step is rejected and the penalty parameter $\rho$ is increased by a factor $c>1$. Note that by definition $\tau^{(j)}$ is a decreasing function which ensures convergence of iterates.


\vspace{-0.2cm}
\section{Numerical Results}
\vspace{-0.2cm}
We now demonstrate the performance of the ADMM-based method on a model optimal control problem. The domain is the unit square $\Omega = [0,1]^2$, discretized using a uniform triangular mesh with mesh size $h = 1/32$ ($2048$ elements) and f is constant in the domain. We consider an elliptic PDE constraint with homogeneous Dirichlet boundary conditions imposed on the west-north boundary $(0, y)$ and $(x, 1)$, and homogeneous Neumann conditions on the other boundaries. The control variables $w(x), v(x)$ are defined piecewise constant on each triangular element over $\Omega$. Figure \ref{fig:controls} and \ref{fig:convergence_plots} represent the results. The choice of initial penalty parameter is instrumental in the convergence, where we see in the right plot in \ref{fig:convergence_plots} that for some large enough $\rho^{(j)}>\bar{\rho}$, the iterates always stay within the funnel/tolerance threshold $\beta\tau^{(j)}$ and the algorithm always converges. When we start with a smaller $\rho$, there are more penalty update steps until a critical value of $\rho$ is reached (as in the $7$-th run in the scatter plot where $\rho$ goes up to $2.4\times10^{-2}$), after which the iterates are acceptable. 
\subsection{Computational complexity and Runtime}
All the computations were performed on the RCI SLURM cluster (2 × Intel Xeon Gold 6146, 46 cores @ 3.2 GHz, 384 GB RAM, 2 TB NVMe). We perform tests on the runtime performance of the algorithm for a range of mesh sizes to show that the randomized solver which only makes moves over connected regions outperforms the deterministic ones like Gurobi which scales exponentially with the dimension of control space. Results are shown in figure \ref{fig:runtimes}. 

\begin{figure}
    \centering
    \includegraphics[width=\linewidth]{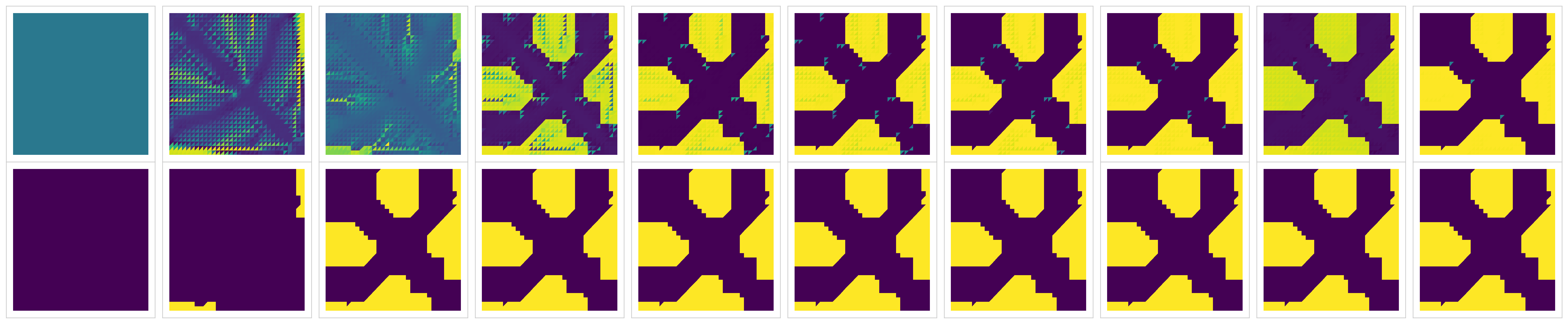}
    \caption{Convergence of $v$ values ({top}) and $w$ values ({bottom}) over iterations (left to right) for $\alpha = 5 \cdot 10^{-5}$, $\rho = 10^{-2}$, $\gamma=2$.}
    \label{fig:controls}
\end{figure}

\begin{figure}
    \centering
    \includegraphics[width=\linewidth]{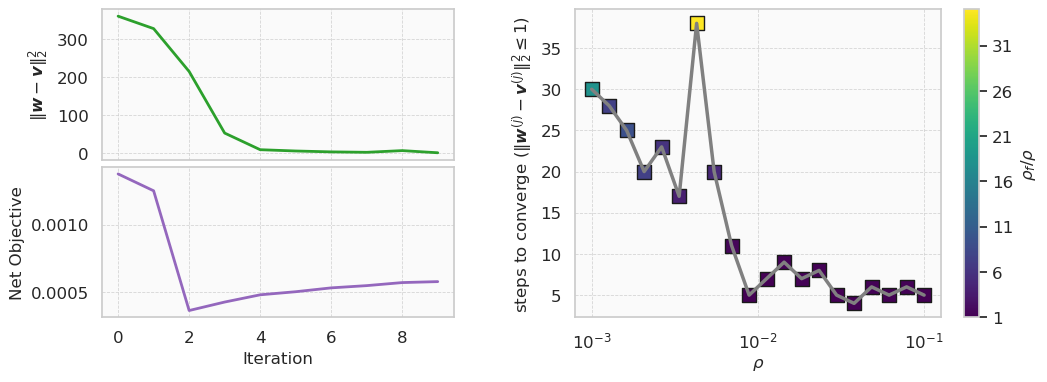}
    \caption{\textbf{Left}: plots show \textbf{Top}: $\|\bw-\bv\|^2_2$ convergence, \textbf{Bottom}: Convergence of the original Objective in \eqref{eq:original_prob}.
     \textbf{Right}: Number of iterations taken by ADMM to converge $\|\bw - \bv\|^2_2 \leq 1$ for different initial $\rho$, where each scatter point represents a unique run. The colorbar represents the ratio of the final to the initial penalty for that run (a larger value means more penalty update steps). For $\rho^{(j)}\gtrsim2.4\times10^{-2}$ the iterates always stay in the funnel and converge.}
    \label{fig:convergence_plots}
\end{figure}
\vspace{-0.2cm}
\begin{figure}
    \centering
    \includegraphics[width=0.9\linewidth]{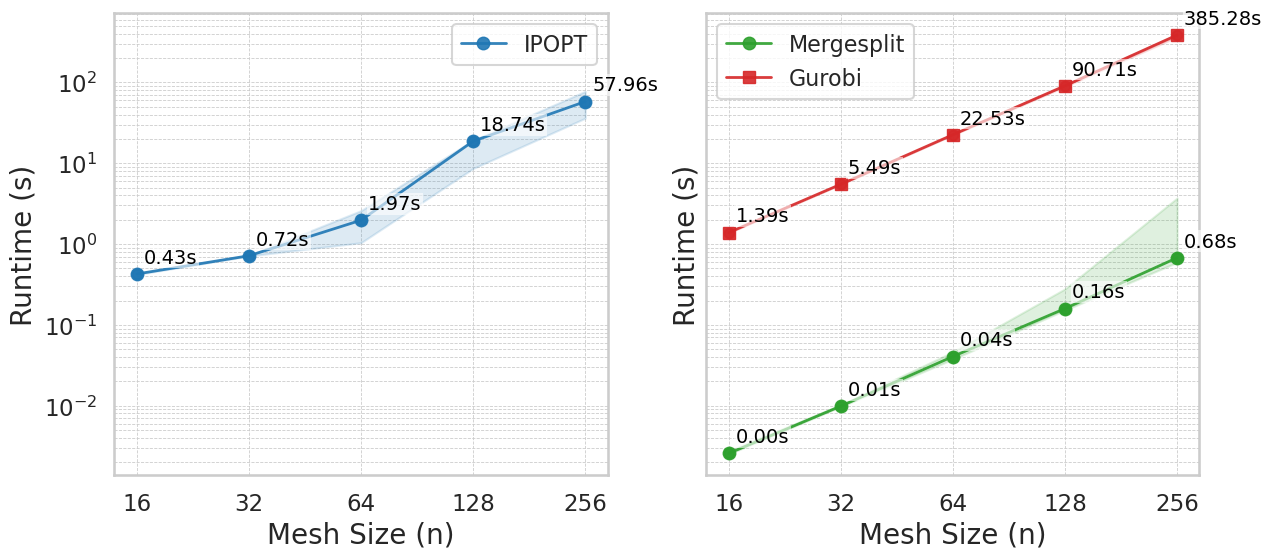}
    \caption{Runtimes for (a) the continuous subproblem solve and (b) a comparison of runtimes for the randomized and method for the integer subproblem solve. The algorithm is run for 40 iterations for each mesh, scatter points represent the median runtime for each run, and the shaded part represents the quartiles. }
    \label{fig:runtimes}
\end{figure}
\subsection*{Acknowledgement}
\vspace{-0.2cm}
HC acknowledges support from GA\v{C}R under Project 24-11664S


\bibliography{admm}

\section{Appendix}\label{s:app}
Here, we will derive asymptotic convergence bounds for the algorithm described in \ref{alg:admm_TO}.
\subsection{Convergence Results}

Consider the Augmented Lagrangian \ref{eq:aug_lagr}:

\begin{align*}
    \Lc_{\rho}(\bu, \bv, \bw, \blam) = \bfu^T \bu
 + \alpha \sum\limits_{e \in E} \sum\limits_{e' \in \Nc(e)} |w_e - w_{e'}| \cdot s_{ee'}
 + \blam(\bw - \bv) + \frac{\rho}{2} \| \bw - \bv \|^2_2
\end{align*}

Let us introduce a shorthand notation:
\begin{itemize}
    \item $\phi(\bv) = \bfu^T \bu$
    \item $\psi(\bw) = \alpha \sum\limits_{e \in E} \sum\limits_{e' \in \Nc(e)} |w_e - w_{e'}| \cdot s_{ee'}$
    \item We will use a shorthand notation for the Augmented Lagrangian as $\Lc(\textbf{x}^{(j)})$ where $x^{(j)} \in \{\bv^{(j)}, \bw^{(j)}, \blam^{(j)}\}$ depending on which variable is being updated
\end{itemize}
\begin{align*}
    \Lc(\bv, \bw, \blam) = \phi(\bv) + \psi(\bw) + \blam(\bw-\bv) + \frac{\rho}{2}\|\bw - \bv\|^2_2
\end{align*}

\begin{assumption}
   \textbf{(Block Lipschitz continuity)} Assume that the Augmented Lagrangian $\Lc(\bv;\overline{\bw},\overline{\blam})$ is a continuously differentiable function over $\RR^{|E|}$ and is block-Lipschitz in $\bv$ for $L_{v}>0$
   \begin{align*}
       \|\Lc(\bv^{(j+1)}) - \Lc(\bv^{(j)})\| \leq L_v\|\bv^{(j+1)} - \bv^{(j)}\|
   \end{align*}
\end{assumption}

\begin{theorem}\label{thm:conv_iters}
    Consider the sequences $\{\tau^{(j)}\}$, $\{\|w^{(j)} - v^{(j)}\|^2_2\}$ and $\Lc^{(j)}$ generated by Algorithm \ref{alg:admm_TO} where $\{\Lc^{(j)}\}$ is bounded below and $\|w^{(j+n)} - v^{(j+n)}\|^2_2\leq \beta \tau^{(j)}$ for all $j,n \in \NN$. Furthermore, let $\beta, \zeta \in (0,1)$ be the constants defined in Algorithm \ref{alg:admm_TO}. In that case, it holds either:
    \begin{itemize}
        \item The penalty parameter is bounded above:$\rho^{(j)}\leq\overline{\rho}\leq +\infty$
        \item Penalty parameter is unbounded:
        $\rho^{(j)} \rightarrow+\infty$
    \end{itemize}
    In both cases it follows that $\|\bw^{(j)} - \bv^{(j)}\|^2_2 \rightarrow 0$ as $j\rightarrow+\infty$
\end{theorem}

\begin{proof}
    We study the two cases, depending on whether or not there is an infinite number of penalty updates. $\rho^{(j+1)}\leftarrow c\rho^{(j)}$ where $c>1$ 

    i) If the number of penalty updates are finite, the funnel update rule implies: 
    \begin{align*}
        \tau^{(j+1)} &= (1-\zeta)\tau^{(j)} + \zeta\|\bw^{(j+1)} - \bv^{(j+1)}\|^2_2\\
        &\leq (1-\zeta)\tau^{(j)} + \zeta \beta \tau^{(j)}\\
        &\leq -\zeta (1-\beta)\tau^{(j)}
    \end{align*}
    Therefore $\tau^{(j+1)}\leq \theta \tau^{(j)}$ where $\theta\overset{def}{=}\zeta(1-\beta)$, thus $\tau^{(j)}\rightarrow 0$ for $j\rightarrow +\infty$.
\vspace{0.5cm}

    ii) Consider that for some $j>J$, let $\overline{\bv} = \bv^{(j)}, \overline{\bw}= \bw^{(j)}$ and $\overline{\blam} = \blam^{(j)}$ be the iterates generated by Algorithm \ref{alg:admm_TO} and let $\rho=\rho^{(j)}$ be the penalty parameter.
The Augmented Lagrangian in \eqref{eq:aug_lagr}:
\begin{equation}\label{eq:lagr_param}
\begin{aligned}
    \Lc(\bv, \bw,\blam) = \bfu^T \bu
 + \alpha \sum\limits_{e \in E} \sum\limits_{e' \in \Nc(e)} |w_e - w_{e'}| \cdot s_{ee'} + \blam(\bw - \bv) +\frac{\rho}{2}\|\bw-\bv\|^2_2
\end{aligned}
\end{equation}
Rewriting the above equation as
\begin{equation}\label{eq:lagr_red}
\begin{aligned}
    \Lc(\bv, \bw,\blam) = F_\bw(\bv)+ \frac{\rho}{2}\|\bw-\bv\|^2_2
\end{aligned}
\end{equation}
where
\begin{align*}
    F_\bw(\bv) \equiv \bfu^T \bu
 + \alpha \sum\limits_{e \in E} \sum\limits_{e' \in \Nc(e)} |w_e - w_{e'}| \cdot s_{ee'} + \blam(\bw-\bv)
\end{align*}

\begin{claim}
    For any $\epsilon > 0$, there exists $\bar{\rho}$ such that if $\rho \ge \bar{\rho}$, then any optimal solution of \eqref{eq:subprob1} satisfies $\|\bv^* - \bw\| < \epsilon$.
\end{claim}

 We now have two terms where the nonlinear term is bounded from above and the quadratic is possibly unbounded as $\rho\rightarrow+\infty$ 

Doing a Taylor expansion of $F_{\bw}(\bv)$ about ${\overline{\bw}}$ in \eqref{eq:lagr_red} we get
\begin{equation}\label{eq:lagr_taylr}
\begin{aligned}
    \Lc({\bv}, \overline{\bw}, \blam) &= F_{\bw}(\overline{\bw}) + \nabla F_\bw(\overline{\bw})^{T}(\bv-\overline{\bw}) + (\bv-\overline{\bw})^T \nabla^2F_{\bw}(\overline{\bw})(\bv-\overline{\bw}) +\frac{\rho}{2}\|\bv-\overline{\bw}\|^2_2
\end{aligned}
\end{equation}
where the terms $\nabla F_{\bw}$ and $\nabla^2F_{\bw}(\overline{\bw})$ are gradient and Hessian w.r.t. the continuous variable.

Let us assume that $\nabla^2F_{\bw} \geq m\textbf{I}$, where $|m| <+\infty$ we can write the above equation as:
\begin{align*}
    \Lc(\bv, \overline{\bw}, \blam) &\geq F_{\bw}(\overline{\bw}) + \nabla F_\bw(\overline{\bw})^{T}(\bv-\overline{\bw}) + m\|\bv-\overline{\bw}\|^2_2 +\frac{\rho}{2}\|\bv-\overline{\bw}\|^2_2
\end{align*}
Using the Cauchy-Schwarz inequality in the above expression,
\begin{equation}\label{eq:lagr_fin_v}
\begin{aligned}
\Lc(\bv, \overline{\bw}, \blam)&\geq F_{\bw}(\overline{\bw}) - \|\nabla F_\bw(\overline{\bw})^{T}\|\|\bv-\overline{\bw}\| + m\|\bv-\overline{\bw}\|^2_2 +\frac{\rho}{2}\|\bv-\overline{\bw}\|^2_2
\end{aligned}
\end{equation}
Now consider the expression in \eqref{eq:lagr_taylr} evaluated at ${\bw}$
\begin{equation}\label{eq:lagr_fin_w}
\begin{aligned}
    \Lc(\bw, \bw, \blam) = F_{\bw}(\overline{\bw})
\end{aligned}
\end{equation}
Assume $\bv^* = \underset{\bv}{\arg \min} \Lc(\bv, \overline{\bw},\overline{\blam})$, subsituting and subtracting \eqref{eq:lagr_fin_v} and \eqref{eq:lagr_fin_w} gives

\begin{equation}\label{eq:lagr_quad}
\begin{aligned}
    \Lc(\bv^{*}, \bw, \blam) - \Lc(\bw, \overline{\bw}, \blam)&\geq - \|\nabla F_\bw(\overline{\bw})^{T}\|\|\bv^*-\overline{\bw}\| + m\|\bv^*-\overline{\bw}\|^2_2 +\frac{\rho}{2}\|\bv^*-\overline{\bw}\|^2_2
\end{aligned}
\end{equation}
Consider the set of points $\bw$ such that
\[
\mathcal{B}_\varepsilon(\bw) := \left\{\, \bw \in \{0,1\}^{|E|} : 0\leq\|\bv^* - {\bw}\| \leq \frac{\|\nabla F_{\bw}({\bw})\|}{(\frac{\rho}{2}+c)} \,\right\}.
\] 
The quadratic term on the rhs of \eqref{eq:lagr_quad} is non-positive within the ball $\mathcal{B}_{\varepsilon}$ of radius $\varepsilon$. 


Now, outside the ball $\mathcal{B_{\varepsilon}}$ the quadratic is positive and strict inequality holds.
\begin{align*}
    \Lc(\bv^*, \bw, \blam) > \Lc(\bw, \bw, \blam)
\end{align*}
 The solution set for this is some $\hat{\mathcal{B}} \subseteq\mathcal{B}_{\varepsilon}.$ If $\rho>\overline{\rho}$, $\|\bv^{*} - \bw\|\leq\frac{\|\nabla F_\bw(\bw)\|}{(\frac{\rho}{2}+c)}$\\

\begin{claim}
    There exists $\bar{\rho} > 0$ such that if $\rho \ge \bar{\rho}$, then any solution of \eqref{eq:subproblem2}, $\bw^*$, minimizes $\|\bv^* - \bw\|_2^2$ over $\mathcal{F}$.
\end{claim}

Argument: Find $\bar{\rho}$ such that this claim holds. We need to show for $\rho$ large enough, that any point $\bw$ which does not minimize $\|\bv^* - \bw\|_2^2$ is not a solution of \eqref{eq:subproblem2}, i.e., $\Lc(\bv^{*}, \bw^{*}, \blam) \le \Lc(\bv^{*}, \bw, \blam)$ where $w^*$ does minimize $\|\bv^* - \bw\|^2$.

Now we analyze the solution for the second subproblem for large $\rho$. The subproblem \eqref{eq:subproblem2}, for each iteration, finds a ${\bw^{*}}$ in the feasible region defined by the constraints such that:
\begin{align*}
    \Lc(\bv^{*}, \bw^{*}, \blam)& - \Lc(\bv^{*}, \bw, \blam) \leq0\\
\bw^{*} \;\in\; \mathcal{F} 
\;\triangleq\;&
\left\{ \bw \in \{0,1\}^{|E|} \;\middle|\;
\sum_{e \in E} w_e \;\leq\; V_{\max}\cdot |E|
\right\}.
\end{align*}
Where:
\begin{align*}
    \Lc(\bv, \bw,\blam) = \bfu^T \bu
 + \alpha \sum\limits_{e \in E} \sum\limits_{e' \in \Nc(e)} |w_e - w_{e'}| \cdot s_{ee'} + \blam(\bw - \bv) +\frac{\rho}{2}\|\bw-\bv\|^2_2
\end{align*}
and let
\begin{align*}
    G(\bw) \equiv \bfu^T \bu
 + \alpha \sum\limits_{e \in E} \sum\limits_{e' \in \Nc(e)} |w_e - w_{e'}| \cdot s_{ee'} + \blam(\bw-\bv)
\end{align*}
Writing out the Augmented Lagrangian at initial $\bw$ and some final $\bw^{*}$ 

\begin{align*}
    G(\bw^{*}) + \frac{\rho}{2}\|\bw^{*} - \bv^{*}\|^2_2 - \left(G(\bw) +\frac{\rho}{2}\|\bw-\bv^{*}\|^2_2\right) \leq 0 \\
    G(\bw^*) - G(\bw) + \frac{\rho}{2}\left(\|\bw^* - \bv^*\|^2 - \|\bw - \bv\|^2 \right) \le 0
\end{align*}

where 
\[
G(\bw) = \alpha \sum\limits_{e \in E} \sum\limits_{e' \in \Nc(e)} |w_e - w_{e'}| \cdot s_{ee'} + \rho\blam(\bw-\bv^{*})
\]

Because $\mathcal{F}$ is finite, $G(\bw^*) - G(\bw)$ is upper bounded by some $\bar{G}$.
\[
    \frac{\rho}{2}\left(\|\bw^* - \bv^*\|^2 - \|\bw - \bv^*\|^2 \right)  \le -\bar{G}
\]
Note $\|\bw - \bv^*\|^2$ is strictly larger than $\|\bw^* - \bv^*\|^2$, $\|\bw - \bv^*\|^2 - \|\bw^* - \bv^*\|^2 \ge m > 0$

\[
    \rho \ge -2\frac{\bar{G}}{\|\bw^* - \bv^*\|^2 - \|\bw - \bv^*\|^2} \ge 2\frac{\bar{G}}{m}
\]

both of the cases above lead to the conclusion: $\|\bv^{(j)} - \bw^{(j)}\|^2_2\rightarrow 0$ as $\rho\rightarrow +\infty$
\end{proof}

\end{document}